\documentclass{amsproc}

\newtheorem{theorem}{Theorem}[section]
\newtheorem{lemma}[theorem]{Lemma}
\newtheorem{prop}[theorem]{Proposition}

\theoremstyle{definition}
\newtheorem{definition}[theorem]{Definition}

\theoremstyle{remark}
\newtheorem{remark}[theorem]{Remark}

\numberwithin{equation}{section}

\begin{document}

\def \Z{\Bbb Z}
\def \C{\Bbb C}
\def \R{\Bbb R}
\def \Q{\Bbb Q}
\def \N{\Bbb N}

\def \A{{\mathcal{A}}}
\def \D{{\mathcal{D}}}
\def \E{{\mathcal{E}}}
\def \E{{\mathcal{E}}}
\def \H{\mathcal{H}}
\def \S{{\mathcal{S}}}
\def \wt{{\rm wt}}
\def \tr{{\rm tr}}
\def \span{{\rm span}}
\def \Res{{\rm Res}}
\def \Der{{\rm Der}}
\def \End{{\rm End}}
\def \Ind {{\rm Ind}}
\def \add {{\rm add}}
\def \Irr {{\rm Irr}}
\def \Aut{{\rm Aut}}
\def \GL{{\rm GL}}
\def \Hom{{\rm Hom}}
\def \mod{{\rm mod}}
\def \ann{{\rm Ann}}
\def \ad{{\rm ad}}
\def \rank{{\rm rank}\;}
\def \<{\langle}
\def \>{\rangle}

\def \g{{\frak{g}}}
\def \h{{\hbar}}
\def \k{{\frak{k}}}
\def \sl{{\frak{sl}}}
\def \gl{{\frak{gl}}}

\newcommand{\m}{\bf m}
\newcommand{\n}{\bf n}
\newcommand{\nno}{\nonumber}
\newcommand{\nord}{\mbox{\scriptsize ${\circ\atop\circ}$}}

\title{Remarks on $\phi$-coordinated modules for quantum vertex algebras}

\author{Haisheng Li}
\address{Department of Mathematical Sciences, Rutgers University, Camden,
New Jersey 08102, and School of Mathematical Sciences,
Xiamen University, Fujian, China }
\email{hli@camden.rutgers.edu}
\thanks{The author was supported in part by China NSF Grants \#11471268, 11571391.}


\subjclass{Primary 17B69, 17B65; Secondary 17B10, 81R10}
\date{January 1, 1994 and, in revised form, June 22, 1994.}

\dedicatory{To James Lepowsky and Robert Wilson on Their 70th Birthdays}

\keywords{Quantum vertex algebra, $\phi$-coordinated module, formal group, associate}

\begin{abstract}
This paper is about $\phi$-coordinated modules for weak quantum vertex algebras. Among the main results, 
several canonical connections among $\phi$-coordinated modules for different $\phi$ are established. 
For vertex operator algebras, a reinterpretation of Frenkel-Huang-Lepowsky's
theorem on contragredient module is given in terms of $\phi$-coordinated modules.
\end{abstract}

\maketitle

\section{Introduction}
Partially motivated by Etingof-Kazhdan's notion of quantum vertex operator algebra (see \cite{EK}), we have developed a theory of (weak) quantum vertex algebras  (see \cite{Li-nonlocal, Li-const, Li-phi, Li-equiv}), where
the notion of (weak) quantum vertex algebra generalizes the notions of vertex algebra, vertex super-algebra, 
and vertex color-algebra (see \cite{xu1}, \cite{xu2}) in a certain natural way.
To associate quantum vertex algebras to quantum affine algebras, in \cite{Li-phi} we developed 
a theory of $\phi$-coordinated (quasi) modules for a weak quantum vertex algebra $V$, 
where $\phi$ is what we called an associate of the one-dimensional additive formal group (law) $F_{\add}(x,y)=x+y$.  

By definition  (cf. \cite{haz}), a general one-dimensional {\em formal group} (over $\C$) is  
a formal power series $F(x,y)\in \C[[x,y]]$, satisfying the conditions
$$F(x,0)=x,\   \   F(0,y)=y, \    \   F(F(x,y),z)=F(x,F(y,z)).$$
As it was pointed out by Borcherds  (see \cite{Bo}), the formal group law $F_{\add}(x,y)$ underlies
the theory of vertex  algebras and modules. This is more apparent from the associativity:
$$Y(u,z+x)Y(v,x)``="Y(Y(u,z)v,x).$$
This is also the case for the theory of (weak) quantum vertex algebras and modules, in which the same associativity 
is postulated though the usual locality (commutativity) is generalized to a braided locality. 
Moreover, formal group law $F_{\add}(x,y)$ also underlies the theory of twisted modules (see \cite{flm}, \cite{ffr}, \cite{dong},
\cite{xu1}, \cite{Li-twisted}, \cite{LTW}).

An {\em associate} of $F_{\add}(x,y)$ (see \cite{Li-phi}) is defined  to be a formal series $\phi(x,z)\in \C((x))[[z]]$, 
satisfying the conditions
\begin{eqnarray*}
\phi(x,0)=x,\    \   \    \   \phi(\phi(x,y),z)=\phi(x,y+z)\ \left(=\phi(x,F_{\add}(y,z))\right).
\end{eqnarray*}
Philosophically speaking, an associate to $F_{\add}(x,y)$ is the same as a $G$-set to a group $G$.
From the definition, the formal group $F_{\add}(x,y)$ itself is an associate.
It was proved in \cite{Li-phi} that for every $p(x)\in \C((x))$, $e^{zp(x)\frac{d}{dx}}(x)$ is an associate 
of $F_{\add}(x,y)$ and every associate can be obtained this way with $p(x)$ uniquely determined.
In particular, taking $p(x)=x^{n+1}$ for $n\in \Z$, we obtain associates
\begin{eqnarray}
\phi_{n}(x,z)=e^{zx^{n+1}\frac{d}{dx}}(x)=\begin{cases}
xe^{z}\   &\mbox{ for }n=0\\
  (x^{-n}-nz)^{-1/n}\ &\mbox{ for }n\ne 0
  \end{cases}
\end{eqnarray}
(cf. \cite{fhl}; (2.6.5)), where 
$$\phi_{-1}(x,z)=e^{z\frac{d}{dx}}(x)=x+z=F_{\add}(x,z).$$
The essence of \cite{Li-phi} is to attach a theory of $\phi$-coordinated modules
for any weak quantum vertex algebra $V$ to each associate $\phi(x,z)$ of $F_{\add}(x,y)$, where
the defining associativity of a $\phi$-coordinated $V$-module $(W,Y_{W})$ can be roughly described as
 \begin{eqnarray}
Y_{W}(u,x_1)Y_{W}(v,x)|_{x_1=\phi(x,z)}``="Y_{W}(Y(u,z)v,x).
\end{eqnarray}

To associate quantum vertex algebras to quantum affine algebras, we have mostly focused on  
$\phi$-coordinated quasi modules  with $\phi(x,z)=xe^{z}$ (see \cite{Li-phi, Li-equiv}). 
Indeed, by using this very theory, weak quantum vertex algebras
have been successfully  associated to quantum affine algebras {\em conceptually,}
 while the construction of an explicit association is in progress. 
 On the other hand, let $V$ be a vertex operator algebra, $(W,Y_{W})$ a $V$-module.
 For $v\in V$, set $X_{W}(v,x)=Y_{W}(x^{L(0)}v,x)$.
 It was shown in \cite{Li-vfa} that $(W,X_{W})$ carries the structure of a
 $\phi$-coordinated module with $\phi(x,z)=xe^{z}$ for the 
new vertex operator algebra obtained by Zhu (see \cite{zhu1, zhu2}) on $V$. 
This result is closely related to a previous result of Lepowsky (see \cite{lep1}). Further investigation  on
$\phi$-coordinated quasi modules with $\phi(x,z)=xe^{z}$ for weak quantum vertex algebras is yet to be done. 

In this paper, we explore canonical connections among  $\phi$-coordinated modules 
for a general weak quantum vertex algebra for various associates $\phi$. Let $\phi(x,z)$ be an associate of
$F_{\add}(x,y)$ and let $f(x)\in x\C[[x]]$ with $f'(0)\ne 0$. 
Set $\phi_{f}(x,z)=f^{-1}(\phi(f(x),z))$. We show that $\phi_{f}(x,z)$ is also an associate of $F_{\add}(x,y)$ and 
that for any weak quantum vertex algebra $V$, the category of $\phi$-coordinated $V$-modules 
is canonically isomorphic to the category of $\phi_{f}$-coordinated $V$-modules.
Furthermore, we show that an associate $\phi(x,z)$ satisfies the condition that $\phi_{f}(x,z)=x+z=F_{\add}(x,z)$ 
for some $f(x)\in x\C[[x]]$ with $f'(0)\ne 0$ if and only if $\phi(x,z)=e^{zp(x)\frac{d}{dx}}(x)$ for $p(x)\in \C[[x]]$ with $p(0)\ne 0$.

In vertex operator algebra theory,  contragredient module due to Frenkel-Huang-Lepowsky (see \cite{fhl}) plays a fundamental role. 
For any module $(W,Y_{W})$ over a vertex operator algebra $V$, the underlying space of
the contragredient module  is the restricted dual $W'=\oplus_{h\in \C}W_{(h)}^{*}$ of $W$, 
where the vertex operator map $Y_{W}'$ is defined by
$$\<Y_{W}'(v,x)\alpha,w\>=\< \alpha, Y_{W}(e^{xL(1)}(-x^{-2})^{L(0)}v,x^{-1})w\>$$
for $v\in V,\  \alpha\in W',\ w\in W$. In this theory,
 the Virasoro algebra, which is part of the vertex operator algebra structure of $V$, is essential.
 A natural question is what one can possibly get if $V$ is a weak quantum vertex algebra, in particular,  a general vertex algebra 
 without a conformal vector. 
 
Note that  the contragredient module theorem of  \cite{fhl} has been generalized by Huang-Lepowsky-Zhang (see \cite{hlz1,hlz2}) to 
 strongly graded conformal vertex algebras.

Let $V$ be a weak quantum vertex algebra, $(W,Y_{W})$ a $V$-module. 
Consider the full dual space $W^{*}$. For $v\in V$, we define
$Y_{W}^{*}(v,x)\in (\End W^{*})[[x,x^{-1}]]$ by
$$\<Y_{W}^{*}(v,x)\alpha,w\>=\<\alpha,Y_{W}(v,x^{-1})w\>$$
for $\alpha\in W^{*},\ w\in W$. Furthermore, we define $D(W)$ to be the subspace of $W^{*}$, 
consisting of each $\alpha$ such that
$$Y_{W}^{*}(v,x)\alpha \in W^{*}((x))\   \   \  \mbox{ for all }v\in V.$$
It is proved in this paper that $D(W)$ equipped with the vertex operator map $Y_{W}^{*}$ is a $\phi$-coordinated module
with $\phi(x,z)=\phi_{1}(x,-z)=\frac{x}{1+zx}$ for the opposite algebra $V^{o}$ (see \cite{sun}; cf. \cite{BK}), where 
the vacuum vector is the same as that of $V$ and the vertex operator map $Y^{o}(\cdot,x)$ is defined by
$$Y^{o}(u,x)v=e^{x\D}Y(v,-x)u\   \   \   \mbox{ for }u,v\in V.$$

Note that when $V$ is a vertex operator algebra, for any $V$-module $W$  
a subspace of $W^{*}$, which was also denoted by  $D(W)$,
was introduced  in \cite{Li-some} and it was proved that $D(W)$ with the vertex operator map $Y_{W}'$
is a weak $V$-module in the sense that $D(W)$ is a module for $V$ viewed as a vertex algebra.
This weak $V$-module structure was used therein to study certain finiteness properties of $W$.
 In fact, in this case the two definitions give the same space $D(W)$ though  the vertex operator maps differ by a twist. 
In this paper,  as a reinterpretation of the contragredient module theorem of \cite{fhl}, we show that 
if $V$ is a vertex operator algebra, the category of modules for $V$ viewed as a vertex algebra is canonically isomorphic to 
the category of $\phi$-coordinated $V$-modules with $\phi(x,z)=\frac{x}{1+zx}$.

This paper is organized as follows: 
In Section 2, we recall the basic results on associates of $F_{\add}(x,y)$ and on $\phi$-coordinated modules 
for a weak quantum vertex algebra, and we
show that the category of $\phi$-coordinated $V$-modules 
is isomorphic to the category of $\phi_{f}$-coordinated $V$-modules.
 In Section 3, we study the dual $D(W)$ of a module $W$ for a weak quantum vertex algebra $V$ and we 
give the equivalence between the categories of modules for $V$ viewed as a vertex algebra
and $\phi$-coordinated $V$-modules with $\phi(x,z)=\frac{x}{1+zx}$ for any vertex operator algebra $V$.

\section{$\phi$-coordinated modules for weak quantum vertex algebras}
In this section, we first recall the notion of associate for the one-dimensional additive formal group (law)
and the notion of $\phi$-coordinated module for a weak quantum vertex algebra $V$ with respect to an associate $\phi$, and we then
show that the category of $\phi$-coordinated $V$-modules 
is isomorphic to the category of $\phi_{f}$-coordinated $V$-modules..

We begin with the one-dimensional additive formal group (law), which is
the formal power series
 $$F_{\add}(x,y)=x+y\in \C[[x,y]].$$  
 (See for example \cite{haz} for the definition of a one-dimensional formal group.)
 
 The following notion was introduced in \cite{Li-phi}:
 
\begin{definition}\label{d-associate}
An {\em associate} of $F_{\add}(x,y)$ is a formal series $\phi(x,z) \in \C((x))[[z]]$ satisfying the conditions
\begin{eqnarray}
\phi(x,0)=x,\    \    \    \  \phi(\phi(x,y),z)=\phi(x,y+z).
\end{eqnarray}
\end{definition}

Note that every associate $\phi(x,z)$ is invertible in $\C((x))[[z]]$ as $\phi(x,0)=x$ is an invertible element of $\C((x))$.
Then $\phi(x,z)^{n}$ is a well defined element of $\C((x))[[z]]$ for every integer $n$.

The following result was also obtained in  \cite{Li-phi}:

\begin{prop}\label{pclassification-associate}
For $p(x)\in \C((x))$, set
$$\phi_{p(x)}(x,z)=e^{zp(x)\frac{d}{dx}}x\in \C((x))[[z]].$$
Then $\phi_{p(x)}(x,z)$ is an associate of $F_{\add}(x,y)$.  
On the other hand, every associate is of this form with $p(x)$ uniquely determined. 
\end{prop}

For $n\in \Z$, set
\begin{eqnarray}
\phi_{n}(x,z)=e^{zx^{n+1}\frac{d}{dx}}\cdot x\in \C((x))[[z]].
\end{eqnarray}
(Note that  the centerless Virasoro algebra, namely the Witt algebra, 
is given by $L(n)=-x^{n+1}\frac{d}{dx}$ for $n\in \Z$.)
In particular, we have
\begin{eqnarray}
\   \    \    \     \    \     \   \   \phi_{-1}(x,z)=x+z=F(x,z),\   \   \   \phi_{0}(x,z)=xe^{z},\  \  \   \phi_{1}(x,z)=\frac{x}{1-zx},
\end{eqnarray}
where as a convention, $\frac{x}{1-zx}$ stands for the formal power series expansion in the nonnegative powers of $z$.


\begin{remark}\label{rA=A}
Let  $\phi(x,z)$ be an associate of $F_{\add}(x,y)$ and let 
$$A(x_1,x_2)\in \Hom (W,W((x_1,x_2)))$$ 
with $W$ a general vector space. We have
\begin{eqnarray}
\   \   \   \    A(\phi(x_2,z),x_2)=A(x_1,x_2)|_{x_1=\phi(x_2,z)}\in \left(\Hom (W,W((x_2)))\right)[[z]].
\end{eqnarray}
On the other hand, for $B(x,z)\in \left(\Hom (W,W((x)))\right)[[z]]$, we have
\begin{eqnarray}
\    \    \    \    B(\phi(x_2,-z),z)=B(x_1,z)|_{x_1=\phi(x_2,-z)}\in \left(\Hom (W,W((x_2)))\right)[[z]].
\end{eqnarray}
Furthermore, for $A(x_1,x_2)\in \Hom (W,W((x_1,x_2)))$, we have
\begin{eqnarray}
\   \   \   \   \left(A(x_1,x_2)|_{x_1=\phi(x_2,z)}\right)|_{x_2=\phi(x_1,-z)}=A(x_1,x_2),
\end{eqnarray}
noticing that
$$\phi(x_2,z)|_{x_2=\phi(x_1,-z)}=\phi(\phi(x_1,-z),z)=\phi(x_1,(-z)+z)=\phi(x_1,0)=x_1.$$
\end{remark}

Let $\phi(x,z)=e^{zp(x)\frac{d}{dx}}\cdot x$ with $p(x)\in \C[x,x^{-1}]$. Noticing that
$\phi(x,z)\in \C[x,x^{-1}][[z]]$, we have $\phi(x^{-1},z)\in \C[x,x^{-1}][[z]]$. As $\phi(x^{-1},0)=x^{-1}$, $\phi(x^{-1},z)$ is 
an invertible element of $\C((x))[[z]]$. Since
$$\phi(x^{-1},z)=e^{-zp(1/x)x^{2}\frac{d}{dx}}\cdot x^{-1},$$
$$1=e^{-zp(1/x)x^{2}\frac{d}{dx}}\cdot (x\cdot x^{-1})=\left(e^{-zp(1/x)x^{2}\frac{d}{dx}}\cdot x\right)
\left(e^{-zp(1/x)x^{2}\frac{d}{dx}}\cdot x^{-1}\right),$$
we get 
$$\frac{1}{\phi(x^{-1},z)}=e^{-zp(1/x)x^{2}\frac{d}{dx}}\cdot x,$$
which is also an associate of $F_{\add}(x,y)$.
Thus we have proved:

\begin{lemma}\label{lphistar}
Suppose $\phi(x,z)=e^{zp(x)\frac{d}{dx}}\cdot x$ with $p(x)\in \C[x,x^{-1}]$. Define 
\begin{eqnarray}
\phi^{*}(x,z)=\frac{1}{\phi(x^{-1},z)}\in \C((x))[[z]].
\end{eqnarray}
Then $\phi^{*}(x,z)$ is also an associate of $F_{\add}(x,y)$ and we have
\begin{eqnarray}
\phi^{*}(x,z)=e^{-zp(1/x)x^{2}\frac{d}{dx}}\cdot x.
\end{eqnarray}
\end{lemma}

It can be readily seen that $(\phi^{*})^{*}=\phi$ for $\phi(x,z)=e^{zp(x)\frac{d}{dx}}\cdot x$ with $p(x)\in \C[x,x^{-1}]$.
 
\begin{remark}\label{rstar-phi}
Let $p(x)=x^{n+1}$ with $n\in \Z$. We have  $-x^{2}p(1/x)=-x^{1-n}$. Then
\begin{eqnarray}
\phi_{n}^{*}(x,z)=\phi_{-n}(x,-z)\   \   \    \mbox{ for }n\in \Z.
\end{eqnarray}
In particular,  we have $\phi^{*}_{-1}(x,z)=\phi_{1}(x,-z)=\frac{x}{1+zx}$ and
$\phi_{0}^{*}(x,z)=\phi_{0}(x,-z)$. 
\end{remark}

Next, we recall  from  \cite{Li-nonlocal} (cf. \cite{EK}) the notion of weak quantum vertex algebra, which generalizes the notions of 
vertex algebra, vertex super-algebra, and vertex color-algebra (see \cite{xu1}, \cite{xu2}) in a certain natural way.
 
\begin{definition}\label{dwqva}
A {\em weak quantum vertex algebra} is a vector space $V$ equipped with a linear map
\begin{eqnarray*}
Y(\cdot,x): &V&\rightarrow (\End V)[[x,x^{-1}]],\\
&v&\mapsto Y(v,x)=\sum_{n\in \Z}v_{n}x^{-n-1}\  \ (\mbox{where }v_{n}\in \End V)
\end{eqnarray*}
and equipped with a distinguished vector ${\bf 1}\in V$, satisfying the conditions that 
$$Y(u,x)v\in V((x))\  \   \mbox{ for }u,v\in V,$$
$$Y({\bf 1},x)v=v,\  \  \  Y(u,x){\bf 1}\in V[[x]]\  \  \mbox{ and }\  \lim_{x\rightarrow 0}Y(u,x){\bf 1}=u,$$
 and that for $u,v\in V$, there exist (finitely many) 
 $$u^{(i)},v^{(i)}\in V,\ f_{i}(x)\in \C((x)) \  \  (i=1,\dots,r)$$
 such that
\begin{eqnarray}\label{eSjacobi}
 &&x_{0}^{-1}\delta\left(\frac{x_1-x_2}{x_0}\right)Y(u,x_1)Y(v,x_2)\nonumber\\
 &&\hspace{1cm}-x_{0}^{-1}\delta\left(\frac{x_2-x_1}{-x_0}\right)\sum_{i}f_{i}(x_2-x_1)Y(v^{(i)},x_2)Y(u^{(i)},x_1)\nonumber\\
 &=&x_{2}^{-1}\delta\left(\frac{x_1-x_0}{x_2}\right)Y(Y(u,x_0)v,x_2).
\end{eqnarray}
\end{definition}
 
By the standard formal variable techniques we immediately have (cf. \cite{Li-nonlocal}):

\begin{prop}\label{pSjacobi=Sloaclity}
 In Definition \ref{dwqva}, the $\S$-Jacobi identity (\ref{eSjacobi}) is equivalent to the property that 
 there exists a nonnegative integer $k$ such that
 \begin{eqnarray}\label{eS-locality}
 && (x_1-x_2)^{k}Y(u,x_1)Y(v,x_2)= (x_1-x_2)^{k}\sum_{i=1}^{r}f_{i}(x_2-x_1)Y(v^{(i)},x_2)Y(u^{(i)},x_1),\nonumber\\
  &&
 \end{eqnarray}
 \begin{eqnarray}
 \left( (x_1-x_2)^{k}Y(u,x_1)Y(v,x_2)\right)|_{x_1=x_2+x_0}=x_0^{k}Y(Y(u,x_0)v,x_2).
 \end{eqnarray}
 \end{prop}
 
Let $V$ be a weak quantum vertex algebra. Denote by $\D$ the linear operator defined by
 \begin{eqnarray}
 \D (v)=\lim_{x\rightarrow 0}\frac{d}{dx} Y(v,x){\bf 1}\   \   \  \mbox{ for }v\in V.
 \end{eqnarray}
 Then (see \cite{Li-nonlocal})
 \begin{eqnarray}
 [\D, Y(v,x)]=Y(\D v,x)=\frac{d}{dx}Y(v,x).
 \end{eqnarray}
 Furthermore, for $u,v\in V$, we have (see \cite{EK})
 \begin{eqnarray}\label{eskew-symmetry}
 Y(u,x)v=e^{x\D}\sum_{i=1}^{r}f_{i}(-x)Y(v^{(i)},-x)u^{(i)},
 \end{eqnarray}
where $u^{(i)},v^{(i)}\in V,\ f_{i}(x)\in \C((x))$ such that (\ref{eS-locality}) holds.
 
\begin{remark}\label{rdual-qva}
 Let $V$ be a weak quantum vertex algebra. Define a linear map
 \begin{eqnarray}
 Y^{o}(\cdot,x):\ V\rightarrow (\End V)[[x,x^{-1}]]
 \end{eqnarray}
 by
 \begin{eqnarray}
 Y^{o}(u,x)v=e^{x\D}Y(v,-x)u\   \   \    \mbox{ for }u,v\in V.
 \end{eqnarray}
Then $(V,Y^{o},{\bf 1})$ is a weak quantum vertex algebra (see \cite{sun}; cf. \cite{BK}). Notice that
if $V$ is a vertex algebra, then $(V,Y^{o},{\bf 1})$ coincides with $(V,Y,{\bf 1})$ as 
$$Y(u,x)v=e^{x\D}Y(v,-x)u=Y^{o}(u,x)v\   \    \  \mbox{ for all }u,v\in V.$$
\end{remark}
 
Note that the last condition in Definition \ref{dwqva} amounts to that there is a linear map
 $$\S(x): \  V\otimes V\rightarrow V\otimes V\otimes \C((x))$$
 such that for $u,v\in V$, (\ref{eSjacobi}) holds with
 $$\S(x)(v\otimes u)=\sum_{i=1}^{r}v^{(i)}\otimes u^{(i)}\otimes f_{i}(x).$$
By {\em a weak quantum vertex algebra with a constant map $\S$} we mean that there is a constant linear map $\S$
such that (\ref{eSjacobi}) holds. Note that here $\S$ does not depend on $x$, but it could depend on vectors $u,v\in V$.

The following notion was introduced in  \cite{Li-phi}:
 
\begin{definition}\label{dphi-module}
Let $V$ be a weak quantum vertex algebra, $\phi$ an associate of $F_{\add}(x,y)$. 
A {\em $\phi$-coordinated $V$-module} is a vector space $W$ equipped with a linear map
\begin{eqnarray*}
Y_{W}(\cdot,x):&V& \rightarrow  (\End W)[[x,x^{-1}]],\\
&v&\mapsto Y_{W}(v,x)
\end{eqnarray*}
satisfying the conditions that 
$$Y_{W}(v,x)w\in W((x))\   \   \   \mbox{ for }v\in V,\ w\in W,  $$
$Y_{W}({\bf 1},x)=1_{W}$, and that for any $u,v\in V$, there exists $k\in \N$ such that
$$(x_1-x_2)^{k}Y_{W}(u,x_1)Y_{W}(v,x_2)\in \Hom(W,W((x_1,x_2))),$$
\begin{eqnarray}\label{ephi-assoc}
&&\left((x_1-x_2)^{k}Y_{W}(u,x_1)Y_{W}(v,x_2)\right)|_{x_1=\phi(x_2,z)}=(\phi(x_2,z)-x_2)^{k}Y_{W}(Y(u,z)v,x_2).\nonumber\\
&&
\end{eqnarray}
\end{definition}

Note that  in Definition \ref{dphi-module}, the first part of the second condition  
guarantees the existence of the substitution on the left-hand side of (\ref{ephi-assoc}).
The following was proved in \cite{Li-phi}:

\begin{lemma}\label{lD-phimodule}
Let $V$ be a weak quantum vertex algebra and let $(W,Y_{W})$ be a $\phi$-coordinated $V$-module. Then
\begin{eqnarray}
Y_{W}(e^{z\D} u,x)= Y_{W}(u,\phi(x,z))\    \    \  \mbox{ for }u\in V.
\end{eqnarray}
\end{lemma}

\begin{remark}\label{rmodule-Jacobi}
In case $\phi(x,z)=x+z$ (the formal group itself), the notion of $\phi$-coordinated $V$-module is equivalent to 
the notion of $V$-module
defined in \cite{Li-nonlocal}. If $(W,Y_{W})$ is a $V$-module, then for $u,v\in V$,
\begin{eqnarray}\label{eSjacobi-module}
 &&x_{0}^{-1}\delta\left(\frac{x_1-x_2}{x_0}\right)Y_{W}(u,x_1)Y_{W}(v,x_2)\nonumber\\
 &&\hspace{1cm}-x_{0}^{-1}\delta\left(\frac{x_2-x_1}{-x_0}\right)\sum_{i}f_{i}(x_2-x_1)Y_{W}(v^{(i)},x_2)Y_{W}(u^{(i)},x_1)\nonumber\\
 &=&x_{2}^{-1}\delta\left(\frac{x_1-x_0}{x_2}\right)Y_{W}(Y(u,x_0)v,x_2),
\end{eqnarray}
where $u^{(i)},v^{(i)},f_{i}(x)$ are given as in (\ref{eSjacobi}). It was proved therein that
\begin{eqnarray}
Y_{W}(\D v,x)=\frac{d}{dx}Y_{W}(v,x)  \   \   \   \mbox{ for }v\in V.
\end{eqnarray}
\end{remark}

Here, we have:

\begin{prop}\label{pphi-property}
Let $V$ be a weak quantum vertex algebra with a constant map $\S$, let $\phi(x,z)$ be an associate of $F_{\add}(x,y)$, and
let $(W,Y_{W})$ be a $\phi$-coordinated $V$-module. Then
for any $u,v\in V$, there exists a nonnegative integer $k$ such that
\begin{eqnarray}
&&(x_1-x_2)^{k}Y_{W}(u,x_1)Y_{W}(v,x_2)=(x_1-x_2)^{k}\sum_{i=1}^{r}Y_{W}(v^{(i)},x_2)Y_{W}(u^{(i)},x_1),\nonumber\\
&&
\end{eqnarray}
where $u^{(i)},v^{(i)}\in V$ such that $Y(u,x)v=\sum_{i=1}^{r}e^{x\D}Y(v^{(i)},-x)u^{(i)}$.
\end{prop}

\begin{proof} Let $u,v\in V$. Since $\S$ is constant, we have $Y(u,x)v=\sum_{i=1}^{r}e^{x\D}Y(v^{(i)},-x)u^{(i)}$ 
for some $u^{(i)},v^{(i)}\in V$.
By definition, there exists a nonnegative integer $k$ such that
\begin{eqnarray*}
&&(x_1-x_2)^{k}Y_{W}(u,x_1)Y_{W}(v,x_2)\in \Hom (W,W((x_1,x_2))),\\
&&\left((x_1-x_2)^{k}Y_{W}(u,x_1)Y_{W}(v,x_2)\right)|_{x_1=\phi(x_2,z)}=(\phi(x_2,z)-x_2)^{k}Y_{W}(Y(u,z)v,x_2),
\end{eqnarray*}
and such that 
\begin{eqnarray*}
&&(x_1-x_2)^{k}Y_{W}(v^{(i)},x_2)Y_{W}(u^{(i)},x_1)\in \Hom (W,W((x_1,x_2))),\\
&&\left((x_1-x_2)^{k}Y_{W}(v^{(i)},x_2)Y_{W}(u^{(i)},x_1)\right)|_{x_2=\phi(x_1,-z)}\\
&=&(x_1-\phi(x_1,-z))^{k}Y_{W}(Y(v^{(i)},-z)u^{(i)},x_1)
\end{eqnarray*}
for $i=1,\dots,r$.
Noticing that 
\begin{eqnarray*}
\sum_{i=1}^{r}Y_{W}(Y(v^{(i)},-z)u^{(i)},x_1)=Y_{W}(e^{-z\D}Y(u,z)v,x_1)
=Y_{W}(Y(u,z)v,\phi(x_1,-z)), 
\end{eqnarray*}
we get 
\begin{eqnarray*}
&&\left((x_1-x_2)^{k}\sum_{i=1}^{r}Y_{W}(v^{(i)},x_2)Y_{W}(u^{(i)},x_1)\right)|_{x_2=\phi(x_1,-z)}\\
&=&(x_1-\phi(x_1,-z))^{k}Y_{W}(Y(u,z)v,\phi(x_1,-z))\\
&=&\left((\phi(x_2,z)-x_2)^{k}Y_{W}(Y(u,z)v,x_2)\right)|_{x_2=\phi(x_1,-z)}\\
&=&\left(\left((x_1-x_2)^{k}Y_{W}(u,x_1)Y_{W}(v,x_2)\right)|_{x_1=\phi(x_2,z)}\right)|_{x_2=\phi(x_1,-z)}\\
&=&(x_1-x_2)^{k}Y_{W}(u,x_1)Y_{W}(v,x_2),
\end{eqnarray*}
where we are using the fact that $\phi(\phi(x_1,-z),z)=x_1$ (recalling Remark \ref{rA=A}).
From this we obtain weak commutativity relation
\begin{eqnarray*}
(x_1-x_2)^{k}Y_{W}(u,x_1)Y_{W}(v,x_2)=(x_1-x_2)^{k}\sum_{i=1}^{r}Y_{W}(v^{(i)},x_2)Y_{W}(u^{(i)},x_1),
\end{eqnarray*}
as desired.
\end{proof}

Next, we show that for some associate $\phi$,  the category of $\phi$-coordinated $V$-modules 
is actually  isomorphic to the category of $V$-modules canonically. Note that for any $f(x)\in x\C[[x]]$ with $f'(0)\ne 0$,
$f(x)$ is invertible with respect to the composition, i.e., there exists $f^{-1}(x)\in x\C[[x]]$ such that $f(f^{-1}(x))=x=f^{-1}(f(x))$.

\begin{lemma}\label{lassociate-f}
Let $\phi(x,z)$ be an associate of $F_{\add}(x,y)$ and let $f(x)\in x\C[[x]]$ with $f'(0)\ne 0$. 
Define
$$\phi_{f}(x,z)=f^{-1}(\phi(f(x),z)).$$
Then $\phi_{f}(x,z)$ is an associate of $F_{\add}(x,y)$.
\end{lemma}

\begin{proof} It is straightforward: As $f(x)\in x\C[[x]]$ with $f'(0)\ne 0$, we have $f^{-1}(z)\in z\C[[z]]$
and $f(x)^{n}\in x^{n}\C[[x]]$ for $n\in \Z$. Thus $\phi(f(x),z)\in \C((x))[[z]]$ and
then $f^{-1}(\phi(f(x),z))\in \C((x))[[z]].$ 
Furthermore, we have
\begin{eqnarray*}
&&\phi_{f}(x,0)=f^{-1}(\phi(f(x),0))=f^{-1}(f(x))=x,\\   
&&\phi_{f}(\phi_{f}(x,y),z)
=f^{-1}(\phi(f(\phi_{f}(x,y)),z)\\
&&\   \   \   \   =f^{-1}(\phi(\phi (f(x),y),z))
=f^{-1}(\phi(f(x),y+z))=\phi_{f}(x,y+z).
\end{eqnarray*}
This proves that $\phi_{f}(x,z)$ is an associate of $F_{\add}(x,y)$.
\end{proof}

\begin{prop}\label{pf-coordinate}
Let $V$ be a weak quantum vertex algebra, let $\phi(x,z)$ be an associate of $F_{\add}(x,y)$, and let
$f(x)\in x\C[[x]]$ with $f'(0)\ne 0$.   Let $(W,Y_{W})$ be a $\phi$-coordinated $V$-module.
For $v\in V$, set
\begin{eqnarray}
Y_{W}^{f}(v,x)=Y_{W}(v,f(x)).
\end{eqnarray}
Then $(W,Y_{W}^{f})$ carries the structure of a $\phi_{f}$-coordinated $V$-module 
where $\phi_{f}(x,z)$ is defined as in Lemma \ref{lassociate-f}. 
\end{prop}

\begin{proof} From definition we have $Y_{W}^{f}({\bf 1},x)=Y_{W}({\bf 1},f(x))=1_{W}$ and
$$Y_{W}^{f}(v,x)w=Y_{W}(v,f(x))w\in W((x))\   \   \   \mbox{ for }v\in V,\ w\in W.$$
Let $u,v\in V$. There exists a nonnegative integer $k$ such that
$$(z_1-z_2)^{k}Y_{W}(u,z_1)Y_{W}(v,z_2)\in \Hom (W,W((z_1,z_2)))$$
and
$$\left((z_1-z_2)^{k}Y_{W}(u,z_1)Y_{W}(v,z_2)\right)|_{z_1=\phi(z_2,x_0)}=(\phi(z_2,x_0)-z_2)^{k}Y_{W}(Y(u,x_0)v,z_2).$$
Then
$$(f(x_1)-f(x_2))^{k}Y_{W}(u,f(x_1))Y_{W}(v,f(x_2))\in \Hom (W,W((x_1,x_2))).$$
As $f(x)\in x\C[[x]]$ with $f'(0)\ne 0$, we have
$$f(x_1)-f(x_2)=(x_1-x_2)g(x_1,x_2),$$
where $g(x_1,x_2)\in \C[[x_1,x_2]]$ with $g(0,0)=f'(0)$. Since $g(0,0)\ne 0$, 
$g(x_1,x_2)$ is an invertible element of $\C[[x_1,x_2]]$.
Then we get
$$(x_1-x_2)^{k}Y_{W}(u,f(x_1))Y_{W}(v,f(x_2))\in \Hom (W,W((x_1,x_2))).$$
The substitution $x_1=\phi_{f}(x_2,x_0)=f^{-1}(\phi(f(x_2),x_0))$ amounts to the substitution $f(x_1)=\phi(f(x_2),x_0)$. Thus
\begin{eqnarray*}
&&g(\phi_{f}(x_2,x_0),x_2)^{k}\left((x_1-x_2)^{k}Y_{W}(u,f(x_1))Y_{W}(v,f(x_2))\right)|_{x_1=\phi_{f}(x_2,x_0)}\\
&=&\left((f(x_1)-f(x_2))^{k}Y_{W}(u,f(x_1))Y_{W}(v,f(x_2))\right)|_{f(x_1)=\phi(f(x_2),x_0)}\\
&=&(\phi(f(x_2),x_0)-f(x_2))^{k}Y_{W}(Y(u,x_0)v,f(x_2))\\
&=&(f(\phi_{f}(x_2,x_0))-f(x_2))^{k}Y_{W}(Y(u,x_0)v,f(x_2))\\
&=&g(\phi_{f}(x_2,x_0),x_2)^{k}(\phi_{f}(x_2,x_0)-x_2)^{k}Y_{W}(Y(u,x_0)v,f(x_2)),
\end{eqnarray*}
which implies
\begin{eqnarray*}
&&\left((x_1-x_2)^{k}Y_{W}(u,f(x_1))Y_{W}(v,f(x_2))\right)|_{x_1=\phi_{f}(x_2,x_0)}\\
&=&(\phi_{f}(x_2,x_0)-x_2)^{k}Y_{W}(Y(u,x_0)v,f(x_2)).
\end{eqnarray*}
This proves that $(W,Y_{W}^{f})$ carries the structure of a $\phi_{f}$-coordinated $V$-module.
\end{proof}

Furthermore, it can be readily seen that if $\theta$ is a homomorphism  of $\phi$-coordinated $V$-modules  
from $(W_1,Y_{W_1})$ to $(W_2,Y_{W_2})$,
then $\theta$ is also a homomorphism of $\phi_{f}$-coordinated $V$-modules from $(W_1,Y^{f}_{W_1})$ to $(W_2,Y^{f}_{W_2})$.  
To summarize we have:

\begin{theorem}\label{tisomorphism}
Let $V$ be a weak quantum vertex algebra, let $\phi(x,z)$ be an associate of $F_{\add}(x,y)$, and let
$f(x)\in x\C[[x]]$ with $f'(0)\ne 0$. Then the category of $\phi$-coordinated $V$-modules is  isomorphic to the category of
$\phi_{f}$-coordinated $V$-modules canonically. 
\end{theorem}

Set $G=\{ f(z)\in z\C[[z]]\ |\  f'(0)\ne 0\}$. Then $G$ is a group with respect to the composition. 
It is straightforward to see that Lemma \ref{lassociate-f} defines a right action of group $G$ 
on the set of associates of $F_{\add}(x,y)$. This gives rise to an equivalence relation on the set of associates of $F_{\add}(x,y)$. 

\begin{lemma}\label{lequivalent}
Let $\phi(x,z)=e^{zp(x)\frac{d}{dx}}(x)$ with $p(x)\in \C((x))$. Then $\phi(x,z)$ is equivalent to $F_{\add}(x,z)$ if and only if
$p(x)\in \C[[x]]$ with $p(0)\ne 0$.
\end{lemma}

\begin{proof} By definition,  $\phi(x,z)$ is equivalent to $F_{\add}(x,z)$ if and only if
 there exists $f(x)\in x\C[[x]]$ with $f'(0)\ne 0$ such that
$\phi(x,z)=f^{-1}(f(x)+z)$, or equivalently, $f(\phi(x,z))=f(x)+z$.
Noticing that
$$f(\phi(x,z))=f\left(e^{zp(x)\frac{d}{dx}}(x)\right)=e^{zp(x)\frac{d}{dx}}f(x),$$
we see that $f(\phi(x,z))=f(x)+z$ if and only if $p(x)\frac{d}{dx}f(x)=1$, i.e., 
$p(x)f'(x)=1$. It is clear that there exists $f(x)\in x\C[[x]]$ with $f'(0)\ne 0$ such that 
$p(x)f'(x)=1$ if and only if  $p(x)\in \C[[x]]$ with $p(0)\ne 0$.
\end{proof}

\begin{remark}\label{rnonequivalence}
It follows from Lemma \ref{lequivalent} that  $\phi_{n}(x,z)$ for $n\ne -1$ is not 
equivalent to $F_{\add}(x,z)=x+z$. In particular, $\phi_{0}(x,z)=xe^{z}$ is {\em not} equivalent to 
$F_{\add}(x,z)$.  
\end{remark}

\begin{remark}\label{rgeneral-equivalence}
Let $p_1(x),p_2(x)\in \C((x))$. One can show that $\phi_{p_1(x)}(x,z)$ is equivalent to $\phi_{p_2(x)}(x,z)$ if and only if
$p_1(x)f'(x)=p_2(f(x))$ for some $f(x)\in x\C[[x]]$ with $f'(0)\ne 0$. But, we are unable to get an explicit characterization.
\end{remark}

\section{$\phi$-coordinated module $(D(W),Y_{W}^{*})$}
In this section, given any module $(W,Y_{W})$ for a weak quantum vertex algebra $V$ with a constant map $\S$, 
we construct a $\phi$-coordinated $V$-module $(D(W),Y_{W}^{*})$ out of $W^{*}$ with $\phi(x,z)=\frac{x}{1+zx}$.
We also show that for a vertex operator algebra $V$, the categories of  $\phi$-coordinated $V$-modules and 
modules for $V$ viewed as a vertex algebra are isomorphic. Note that a $V$-module is the same as a $\phi_{-1}$-coordinated 
$V$-module while $\frac{x}{1+zx}=\phi_1(x,-z)=\phi^{*}_{-1}(x,z)$.

First, we show that a certain Jacobi identity holds on $\phi$-coordinated $V$-modules with $\phi(x,z)=\frac{x}{1+zx}$
 just as with usual $V$-modules.
(A Jacobi identity on $\phi$-coordinated modules with $\phi(x,z)=xe^{z}$ was obtained in \cite{Li-phi}.)

\begin{prop}\label{pphi1}
Let $V$ be a weak quantum vertex algebra with a constant map $\S$ and
let $\phi(x,z)=\phi_{1}(x,-z)=\frac{x}{1+zx}$. 
Then, in the definition of a $\phi$-coordinated $V$-module, the second condition is equivalent to
\begin{eqnarray}\label{phi1-jacobi}
&&z^{-1}\delta\left(\frac{x_2^{-1}-x_1^{-1}}{z}\right)Y_{W}(u,x_{1})Y_{W}(v,x_2)\nonumber\\
&&\hspace{1cm} -z^{-1}\delta\left(\frac{x_1^{-1}-x_2^{-1}}{-z}\right)\sum_{i=1}^{r}Y_{W}(v^{(i)},x_2)Y_{W}(u^{(i)},x_1)\nonumber\\
&=&x_2\delta\left(\frac{x_1^{-1}+z}{x_2^{-1}}\right)Y_{W}(Y(u,-z)v,x_2)
\end{eqnarray}
for $u,v\in V$, where $u^{(i)},v^{(i)}\in V$ such that $Y(u,x)v=\sum_{i=1}^{r}e^{x\D}Y(v^{(i)},-x)u^{(i)}$.
\end{prop}

\begin{proof} First, assume that  (\ref{phi1-jacobi}) holds for any $u,v\in V$. 
Let $u,v\in V$. There exists a nonnegative integer $k$ such that $x^{k}Y(u,x)v\in V[[x]]$. 
Applying $\Res_{z}z^{k}x_1^{k}x_2^{k}$ to (\ref{phi1-jacobi}) we get
\begin{eqnarray}\label{elocality-k}
\   \   \   \   \    \   \   \    (x_1-x_2)^{k}Y_{W}(u,x_{1})Y_{W}(v,x_2)=(x_1-x_2)^{k}\sum_{i=1}^{r}Y_{W}(v^{(i)},x_2)Y_{W}(u^{(i)},x_{1}),
\end{eqnarray}
which implies
$$(x_1-x_2)^{k}Y_{W}(u,x_{1})Y_{W}(v,x_2)\in \Hom (W,W((x_1,x_2))).$$
Multiplying both sides of (\ref{phi1-jacobi}) by $z^{k}x_1^{k}x_2^{k}$, and 
then using (\ref{elocality-k}) and the basic delta-function properties  we get
\begin{eqnarray*}
&&x_2\delta\left(\frac{x_1^{-1}+z}{x_2^{-1}}\right)\left( (x_1-x_2)^{k}Y_{W}(u,x_{1})Y_{W}(v,x_2)\right)\\
&=&x_2\delta\left(\frac{x_1^{-1}+z}{x_2^{-1}}\right)(zx_1x_2)^{k}Y_{W}(Y(u,-z)v,x_2),
\end{eqnarray*}
which can also be written as
\begin{eqnarray*}
&&x_1\delta\left(\frac{x_2^{-1}-z}{x_1^{-1}}\right)\left( (x_1-x_2)^{k}Y_{W}(u,x_{1})Y_{W}(v,x_2)\right)\\
&=&x_1\delta\left(\frac{x_2^{-1}-z}{x_1^{-1}}\right)(zx_1x_2)^{k}Y_{W}(Y(u,-z)v,x_2).
\end{eqnarray*}
Then applying $\Res_{x_1}x_1^{-2}$ we get
\begin{eqnarray*}
\left( (x_1-x_2)^{k}Y_{W}(u,x_{1})Y_{W}(v,x_2)\right)|_{x_1^{-1}=x_2^{-1}-z}
=\left(\frac{zx_2}{x_2^{-1}-z}\right)^{k}Y_{W}(Y(u,-z)v,x_2).
\end{eqnarray*}
That is,
\begin{eqnarray*}
\left( (x_1-x_2)^{k}Y_{W}(u,x_{1})Y_{W}(v,x_2)\right)|_{x_1=\frac{x_2}{1-zx_2}}
=\left(\frac{x_2}{1-zx_2}-x_2\right)^{k}Y_{W}(Y(u,-z)v,x_2).
\end{eqnarray*}
Thus
\begin{eqnarray*}
\left( (x_1-x_2)^{k}Y_{W}(u,x_{1})Y_{W}(v,x_2)\right)|_{x_1=\frac{x_2}{1+zx_2}}
=\left(\frac{x_2}{1+zx_2}-x_2\right)^{k}Y_{W}(Y(u,z)v,x_2).
\end{eqnarray*}
 This proves that $(W,Y_{W})$ is a $\phi$-coordinated $V$-module.
 
On the other hand, assume that $(W,Y_{W})$ is a $\phi$-coordinated $V$-module. 
Let $u,v\in V$, and let $u^{(i)},v^{(i)}\in V$, $1\le i\le r$ such that $Y(u,x)v=\sum_{i=1}^{r}e^{x\D}Y(v^{(i)},-x)u^{(i)}$.
By Proposition \ref{pphi-property}, there exists a nonnegative integer $k$ such that
\begin{eqnarray}
\   \   \   \   \   \   \   \   (x_1-x_2)^{k}Y_{W}(u,x_1)Y_{W}(v,x_2)=(x_1-x_2)^{k}\sum_{i=1}^{r}Y_{W}(v^{(i)},x_2)Y_{W}(u^{(i)},x_1).
\end{eqnarray}
As this implies $(x_1-x_2)^{k}Y_{W}(u,x_1)Y_{W}(v,x_2)\in \Hom (W,W((x_1,x_2)))$, we have
\begin{eqnarray*}
\left((x_1-x_2)^{k}Y_{W}(u,x_1)Y_{W}(v,x_2)\right)|_{x_1=\phi(x_2,z)}=\left(\frac{x_2}{1+zx_2}-x_2\right)^{k}Y_{W}(Y(u,z)v,x_2),
\end{eqnarray*}
which gives
\begin{eqnarray*}
\left( (x_1-x_2)^{k}Y_{W}(u,x_{1})Y_{W}(v,x_2)\right)|_{x_1^{-1}=x_2^{-1}-z}
=\left(\frac{zx_2}{x_2^{-1}-z}\right)^{k}Y_{W}(Y(u,-z)v,x_2).
\end{eqnarray*}
Note that
$$z^{-1}\delta\left(\frac{x_2^{-1}-x_1^{-1}}{z}\right) -z^{-1}\delta\left(\frac{x_1^{-1}-x_2^{-1}}{-z}\right)
=x_2\delta\left(\frac{x_1^{-1}+z}{x_2^{-1}}\right).$$
Then Jacobi identity (\ref{phi1-jacobi}) follows.
\end{proof}

\begin{definition}\label{dYstar}
Let $V$ be a weak quantum vertex algebra and let $(W,Y_{W})$ be a $V$-module. 
Define a linear map
\begin{eqnarray*}
Y_{W}^{*}(\cdot,x): &V&\rightarrow (\End W^{*})[[x,x^{-1}]],\\
&v&\mapsto Y_{W}^{*}(v,x)
\end{eqnarray*}
by
\begin{eqnarray}
\<Y_{W}^{*}(v,x)\alpha, w\>=\<\alpha, Y_{W}(v,x^{-1})w\>
\end{eqnarray}
for $v\in V,\ \alpha\in W^{*},\ w\in W$. 
\end{definition}

\begin{lemma}\label{lY*D}
We have
\begin{eqnarray}
Y_{W}^{*}(\D v,x)\alpha=-x^{2}\frac{d}{dx}Y_{W}^{*}(v,x)\alpha
\end{eqnarray}
for $v\in V,\ \alpha\in W^{*}$.
\end{lemma}
 
\begin{proof} It is straightforward: For any $w\in W$, we have
\begin{eqnarray*}
&&\<Y_{W}^{*}(\D v,x)\alpha,w\>=\<\alpha,Y_{W}(\D v,x^{-1})w\>=\frac{d}{dx^{-1}}\<\alpha,Y_{W}(v,x^{-1})w\>\\
&=&-x^{2}\frac{d}{dx}\<\alpha,Y_{W}(v,x^{-1})w\>
=-x^{2}\frac{d}{dx}\<Y_{W}^{*}(v,x)\alpha,w\>,
\end{eqnarray*}
as needed. 
\end{proof}

\begin{definition}\label{dDual-W}
Let $V$ be a weak quantum vertex algebra, $(W,Y_{W})$ a $V$-module. 
Define $D(W)$ to be the set of $\alpha\in W^{*}$ satisfying the condition 
\begin{eqnarray}
Y_{W}^{*}(v,x)\alpha\in W^{*}((x))\   \   \   \mbox{ for every }v\in V.
\end{eqnarray}
\end{definition}

It is clear that $D(W)$  is a subspace of $W^{*}$. Furthermore, we have:

\begin{theorem}\label{tmain1}
Let $V$ be a weak quantum vertex algebra with a constant map $\S$ and let $(W,Y_{W})$ be a $V$-module. Then
\begin{eqnarray}
Y_{W}^{*}(v,x)\alpha\in D (W)((x))\   \   \    \mbox{ for }v\in V,\  \alpha\in D(W).
\end{eqnarray}
Furthermore, the pair $(D(W),Y_{W}^{*})$ carries the structure of a $\phi$-coordinated $V^{o}$-module with $\phi(x,z)=\frac{x}{1+zx}$.
\end{theorem}

\begin{proof} Let $u,v\in V,\ \alpha\in D(W)$. From the $\S$-Jacobi identity for $Y_{W}$ we have
\begin{eqnarray}\label{edual-jacobi-0}
&&x_0^{-1}\delta\left(\frac{x_1^{-1}-x_2^{-1}}{x_0}\right)Y_{W}^{*}(v,x_2)Y_{W}^{*}(u,x_1)\alpha\nonumber\\
&&\hspace{1cm} -x_0^{-1}\delta\left(\frac{x_2^{-1}-x_1^{-1}}{-x_0}\right)
\sum_{i=1}^{r}Y_{W}^{*}(u^{(i)},x_{1})Y_{W}^{*}(v^{(i)},x_2)\alpha\nonumber\\
&=&x_2\delta\left(\frac{x_1^{-1}-x_0}{x_2^{-1}}\right)Y_{W}^{*}(Y(u,x_0)v,x_2)\alpha,
\end{eqnarray}
where $u^{(i)},v^{(i)}$ $(i=1,\dots,r)$ are vectors in $V$ such that (\ref{eSjacobi-module}) holds with $f_{i}(x)=1$.
Let $n\in \Z$. Applying $\Res_{x_0}\Res_{x_1}x_1^{n}$ to (\ref{edual-jacobi-0}), we get
\begin{eqnarray*}
&&Y_{W}^{*}(v,x_2)\Res_{x_1}x_1^{n}Y_{W}^{*}(u,x_1)\alpha 
-\Res_{x_1}x_1^{n}\sum_{i=1}^{r}Y_{W}^{*}(u^{(i)},x_{1})Y_{W}^{*}(v^{(i)},x_2)\alpha\   \   \nonumber\\
&=&\Res_{x_0}(x_2^{-1}+x_0)^{-n-2}Y_{W}^{*}(Y(u,x_0)v,x_2)\alpha\\
&=&\sum_{j\ge 0}\binom{-n-2}{j}x_2^{n+2+j}Y_{W}^{*}(u_{j}v,x_2)\alpha.
\end{eqnarray*}
That is,
\begin{eqnarray*}
&&Y_{W}^{*}(v,x_2)\Res_{x_1}x_1^{n}Y_{W}^{*}(u,x_1)\alpha \nonumber\\
&=&\sum_{j\ge 0}\binom{-n-2}{j}x_2^{n+2+j}Y_{W}^{*}(u_{j}v,x_2)\alpha+
\Res_{x_1}x_1^{n}\sum_{i=1}^{r}Y_{W}^{*}(u^{(i)},x_{1})Y_{W}^{*}(v^{(i)},x_2)\alpha.
\end{eqnarray*}
Since $\alpha\in D(W)$ and $u_{j}v=0$ for $j$ sufficiently large, we get
$$Y_{W}^{*}(v,x_2)\Res_{x_1}x_1^{n}Y_{W}^{*}(u,x_1)\alpha \in W^{*}((x_2)).$$
This proves $Y_{W}^{*}(u,x)\alpha\in D (W)((x))$ for any $u\in V,\ n\in \Z$.

In the first part, we have proved that $Y_{W}^{*}(\cdot,x)$ gives rise to a linear map from $V$ to 
$\Hom (D(W),D(W)((x)))$.
It is clear that $Y_{W}^{*}({\bf 1},x)=1_{W^{*}}$. 
Furthermore, let $u,v\in V,\ \alpha \in D(W)$. Assume that $\bar{u}^{(j)},\bar{v}^{(j)}\in V,\ j=1,\dots, s$ such that
$$Y(v,x)u=\sum_{j=1}^{s}e^{x\D}Y(\bar{u}^{(j)},-x)\bar{v}^{(j)}.$$
Then using Lemma \ref{lY*D} we obtain 
\begin{eqnarray}\label{edual-jacobi-3}
&&x_0^{-1}\delta\left(\frac{x_2^{-1}-x_1^{-1}}{x_0}\right)Y_{W}^{*}(u,x_1)Y_{W}^{*}(v,x_2)\alpha\nonumber\\
&&\hspace{1cm} -x_0^{-1}\delta\left(\frac{x_1^{-1}-x_2^{-1}}{-x_0}\right)
\sum_{j=1}^{s}Y_{W}^{*}(\bar{u}^{(j)},x_{2})Y_{W}^{*}(\bar{v}^{(j)},x_1)\alpha\nonumber\\
&=&x_1\delta\left(\frac{x_2^{-1}-x_0}{x_1^{-1}}\right)Y_{W}^{*}(Y(v,x_0)u,x_1)\alpha\nonumber\\
&=&x_1\delta\left(\frac{x_2^{-1}-x_0}{x_1^{-1}}\right)Y_{W}^{*}(e^{x_0\D}Y^{o}(u,-x_0)v,x_1)\alpha\nonumber\\
&=&x_1\delta\left(\frac{x_2^{-1}-x_0}{x_1^{-1}}\right)e^{-x_0x_1^{2}\frac{\partial}{\partial x_1}}
Y_{W}^{*}(Y^{o}(u,-x_0)v,x_1)\alpha\nonumber\\
&=&x_1\delta\left(\frac{x_2^{-1}-x_0}{x_1^{-1}}\right)
Y_{W}^{*}\left(Y^{o}(u,-x_0)v,\frac{1}{x_1^{-1}+x_0}\right)\alpha\nonumber\\
&=&x_1\delta\left(\frac{x_2^{-1}-x_0}{x_1^{-1}}\right)Y_{W}^{*}(Y^{o}(u,-x_0)v,x_2)\alpha,
\end{eqnarray}
noticing that
$$e^{-zx^{2}\frac{d}{dx}}x=\phi_{1}(x,-z)=\frac{x}{1+zx}=\frac{1}{x^{-1}+z}.$$
Then by Proposition \ref{pphi1} we conclude that
$(D(W),Y_{W}^{*})$ carries the structure of a $\phi$-coordinated $V^{o}$-module.
\end{proof}

\begin{remark}\label{tmain2}
Assume $\phi(x,z)=e^{zp(x)\frac{d}{dx}}\cdot x$ with $p(x)\in \C[x,x^{-1}]$.
Let $(W,Y_{W})$ be a $\phi$-coordinated $V$-module.
We believe that it is also true that $(D(W),Y_{W}^{*})$ is a $\phi^{*}$-coordinated $V^{o}$-module 
with $\phi^{*}(x,z)$ defined as in Lemma \ref{lphistar}. What we lack is a Jacobi identity for $\phi$-coordinated $V$-modules.
\end{remark}

\begin{remark}\label{rold}
Let $V$ be a vertex operator algebra and let $(W,Y_{W})$ be a module for
 $V$ viewed as a vertex algebra.
For $v\in V$, following \cite{fhl} define $Y_{W}'(v,x)\in (\End W^{*})[[x,x^{-1}]]$ by
$$\< Y_{W}'(v,x)\alpha,w\>=\<\alpha,Y_{W}\left(e^{xL(1)}(-x^{-2})^{L(0)}v,x\right)w\>$$
for $\alpha\in W^{*},\ w\in W$. It can be readily seen that
$$D(W)=\{ \alpha\in W^{*}\ |\  Y_{W}'(v,x)\alpha\in W^{*}((x)) \   \mbox{ for all }v\in V\}.$$
It was proved in \cite{Li-some} that $(D(W),Y_{W}')$ is a module for $V$ viewed as a vertex algebra.
\end{remark}

The following is an interpretation of Theorem 5.2.1 of \cite{fhl} in terms of $\phi$-coordinated modules:

\begin{prop}\label{pfhl-dual}
Let $V$ be a vertex operator algebra and let $W$ be a vector space equipped with a linear map
$$Y_{W}(\cdot,x): V\rightarrow \Hom (W,W((x)))\subset (\End W)[[x,x^{-1}]]$$
such that $Y_{W}({\bf 1},x)=1_{W}$. For $v\in V$, set
$$Y_{W}^{new}(v,x)=Y_{W}\left(e^{xL(1)}(-x^{-2})^{L(0)}v,x\right).$$
Then $(W,Y_{W})$ is a $\phi$-coordinated $V$-module with $\phi(x,z)=\frac{x}{1+zx}$
if and only if $(W,Y_{W}^{new})$ is a module for $V$ viewed as a vertex algebra. 
Furthermore, this gives rise to an isomorphism between the categories of  $\phi$-coordinated $V$-modules and $V$-modules
with $V$ viewed as a vertex algebra.
\end{prop}

\begin{proof} It basically follows from the arguments in \cite{fhl}. 
By Proposition \ref{pphi1},  $(W,Y_{W})$ is  a $\phi$-coordinated $V$-module if and only if
\begin{eqnarray*}
&&z^{-1}\delta\left(\frac{x_2^{-1}-x_1^{-1}}{z}\right)Y_{W}(u,x_1)Y_{W}(v,x_2)
-z^{-1}\delta\left(\frac{x_1^{-1}-x_2^{-1}}{-z}\right)Y_{W}(v,x_2)Y_{W}(u,x_1)\   \    \\
&&\hspace{2cm}=x_2\delta\left(\frac{x_1^{-1}+z}{x_2^{-1}}\right)Y_{W}(Y(u,-z)v,x_2)
\end{eqnarray*}
for all $u,v\in V$, which is
\begin{eqnarray*}
&&z^{-1}\delta\left(\frac{x_1-x_2}{zx_1x_2}\right)Y_{W}(u,x_1)Y_{W}(v,x_2)
-z^{-1}\delta\left(\frac{x_2-x_1}{-zx_1x_2}\right)Y_{W}(v,x_2)Y_{W}(u,x_1)\   \    \\
&&\hspace{2cm}=x_2\delta\left(\frac{x_2+zx_1x_2}{x_1}\right)Y_{W}(Y(u,-z)v,x_2).
\end{eqnarray*}
Setting $x_0=zx_1x_2$, we get
\begin{eqnarray}\label{esecond}
&&x_0^{-1}\delta\left(\frac{x_1-x_2}{x_0}\right)Y_{W}(u,x_1)Y_{W}(v,x_2)
-x_0^{-1}\delta\left(\frac{x_2-x_1}{-x_0}\right)Y_{W}(v,x_2)Y_{W}(u,x_1)\   \    \nonumber\\
&&\hspace{2cm}=x_1^{-1}\delta\left(\frac{x_2+x_0}{x_1}\right)Y_{W}\left(Y\left(u,-x_0x_1^{-1}x_2^{-1}\right)v,x_2\right)\nonumber\\
&&\hspace{2cm}=x_1^{-1}\delta\left(\frac{x_2+x_0}{x_1}\right)Y_{W}\left(Y\left(u,-\frac{x_0}{(x_2+x_0)x_2}\right)v,x_2\right)
\end{eqnarray}
for all $u,v\in V$. Set
$$\Phi(x)=e^{xL(1)}(-x^{-2})^{L(0)}.$$
We see that (\ref{esecond}) is equivalent to 
\begin{eqnarray}\label{ethird}
\  \  &&x_0^{-1}\delta\left(\frac{x_1-x_2}{x_0}\right)Y_{W}(\Phi(x_1)u,x_1)Y_{W}(\Phi(x_2)v,x_2)\nonumber\\
&&\   \    \   \   \   -x_0^{-1}\delta\left(\frac{x_2-x_1}{-x_0}\right)Y_{W}(\Phi(x_2)v,x_2)Y_{W}(\Phi(x_1)u,x_1)\nonumber\\
&=&x_1^{-1}\delta\left(\frac{x_2+x_0}{x_1}\right)Y_{W}\left(Y\left(\Phi(x_1)u,-x_0x_1^{-1}x_2^{-1}\right)\Phi(x_2)v,x_2\right)\nonumber\\
&=&x_1^{-1}\delta\left(\frac{x_2+x_0}{x_1}\right)Y_{W}\left(Y\left(\Phi(x_2+x_0)u,-\frac{x_0}{(x_2+x_0)x_2}\right)\Phi(x_2)v,x_2\right)
\end{eqnarray}
for all $u,v\in V$. The formula (5.2.36) in  \cite{fhl} states 
\begin{eqnarray}
Y\left(\Phi(x_2+x_0)u,-\frac{x_0}{(x_2+x_0)x_2}\right)\Phi(x_2)v=\Phi(x_2)Y(u,x_0)v.
\end{eqnarray}
Then (\ref{ethird}) is equivalent to
\begin{eqnarray*}
&&x_0^{-1}\delta\left(\frac{x_1-x_2}{x_0}\right)Y_{W}^{new}(u,x_1)Y_{W}^{new}(v,x_2)\nonumber\\
&&\  \   \   \  -x_0^{-1}\delta\left(\frac{x_2-x_1}{-x_0}\right)Y_{W}^{new}(v,x_2)Y_{W}^{new}(u,x_1)\nonumber \\
&=&x_1^{-1}\delta\left(\frac{x_2+x_0}{x_1}\right)Y_{W}^{new}(Y(u,x_0)v,x_2)
\end{eqnarray*}
for all $u,v\in V$. Consequently, $(W,Y_{W})$ is  a $\phi$-coordinated $V$-module if and only if
$(W,Y_{W}^{new})$ is a module for $V$ viewed as a vertex algebra.
The ``furthermore'' assertion is clear.
\end{proof}

\bibliographystyle{amsalpha}

\end{document}